\documentclass[10pt]{article}
\usepackage[a4paper, total={4.5in, 7.125in}]{geometry}
\usepackage{amsthm}
\usepackage{amsmath}
\usepackage{amssymb}

\title{Enumerating $k$th Roots in the Symmetric Inverse Monoid}
\author{Christopher W. York}

\begin{document}

\newtheorem{definition}{Definition}
\newtheorem{theorem}{Theorem}
\newtheorem{lemma}{Lemma}
\newtheorem{Corollary}{Corollary}
\newtheorem{proposition}{Proposition}
\newtheorem{remark}{Remark}
\newcommand{\freq}{\mathop{\mathrm{freq}}}
\newcommand{\SIM}{\mathop{\mathrm{SIM}}}

\maketitle

\begin{abstract}
The symmetric inverse monoid, $\SIM(n)$, is the set of all partial one-to-one mappings from the set $\{1,2,\cdots,n\}$ to itself under the operation of composition. Earlier research on the symmetric inverse monoid delineated the process for determining whether an element of $\SIM(n)$ has a $k$th root. The problem of enumerating $k$th roots of a given element of $\SIM(n)$ has since been posed, which is solved in this work. In order to find the number of $k$th roots of an element, all that is needed is to know the cycle and path structure of the element. Conveniently, the cycle and cycle-free components may be considered separately in calculating the number of $k$th roots. Since the enumeration problem has been completed for the symmetric group, this paper only focuses on the cycle-free elements of $\SIM(n)$. The formulae derived for cycle-free elements of $\SIM(n)$ here utilize integer partitions, similar to their use in the expressions given for the number of $k$th roots of permutations.
\end{abstract}

\section{Introduction}
\label{intro}
The symmetric inverse monoid over a set $X$, denoted by $\SIM(X)$, is the set of all partial one-to-one mappings from $X$ to itself. When denoted by $\SIM(n)$, with $n$ a positive integer, $\SIM(n)$ refers to the set of all partial one-to-one mappings from $\{1,2,\cdots,n\}$ to itself. For instance, the element $\left(\begin{array}{cccc}1 & 2 & 3 & 4 \\ - & 2 & 4 & 3\end{array}\right)\in\SIM(n)$ maps 1 to nothing, 2 to itself, 3 to 4, and 4 to 3. Under function composition, $\SIM(n)$ forms an {\em inverse semigroup}, which is a semigroup $S$ such that for all $a\in S$, there exists $b\in S$ such that $a=aba$ and $b=bab$, and all idempotents of $S$ commute \cite{kn:Law}.

Previously, Annin et al. determined necessary and sufficient conditions for an arbitrary element $\sigma\in\SIM(n)$ to possess a $k$th root \cite{kn:Ann2}, and prior to this, they solved this problem for the symmetric and alternating groups \cite{kn:Ann1}. At the end of \cite{kn:Ann2}, they posed some interesting questions. Given an element of $\SIM(n)$, exactly how many $k$th roots does it possess? Also, what is the probability that a random element of $\SIM(n)$ will possess a $k$th root? This paper seeks to solve the former problem.

Not suprisingly, the symmetric group of order $n$, $S_n$, is a subsemigroup of $\SIM(n)$. For the purposes of this paper, let $S_A$ denote the set of permutations over the set $A$. Of crucial importance is that, as Annin et al. cite \cite{kn:Ann2}, each element of $\SIM(n)$ can be uniquely expressed as the product of disjoint cycles and paths, which are defined in the next section.

\section{Preliminary Results}

\begin{definition}
The {\em path} $[a_1a_2a_3\cdots a_n]$, where $a_1, a_2, \cdots, a_n$ are distinct, is the mapping from $a_1$ to $a_2$, $a_2$ to $a_3$, $\cdots$, $a_{n-1}$ to $a_n$, and $a_n$ to nothing.
\end{definition}

\begin{lemma}
Every element of $\SIM(n)$ can be expressed as the product of disjoint paths and cycles.
\end{lemma}

\begin{definition}
An element $\alpha\in\SIM(n)$ is a {\em single cycle} or {\em single path} element if $\alpha$ can be expressed as either $\alpha = [a_1a_2\cdots a_\ell]$ or $\alpha = (a_1a_2\cdots a_\ell)$ where $a_1, a_2, \cdots, a_\ell$ are distinct. The {\em length} of $\alpha$, denoted by $\ell(\alpha)$, equals $\ell$, and $a_1, a_2, \cdots, a_\ell$ are said to be {\em digits} of the cycle or path $\alpha$.
\end{definition}

\begin{definition}
A {\em $k$-interlacing} of the disjoint paths and/or cycles \sloppy{$\sigma_1, \sigma_2, \cdots, \sigma_n$} is a single cycle or path element $\alpha$ such that $\alpha^k = \sigma_1\sigma_2\cdots\sigma_n$, and every digit of $\alpha$ is in one of $\sigma_1,\sigma_2,\cdots,\sigma_n$, and every digit of $\sigma_1, \sigma_2, \cdots, \sigma_n$ is in $\alpha$.
\end{definition}

{\bf Remark.} Whenever mentioned, disjoint cycle decompositions will include the cycles of length one.

{\bf Examples.} The path [14725836] is a 3-interlacing of [123], [456], and [78] since the paths have the same digits as [14725836] and $[14725836]^3=[123][456][78]$. Likewise, the cycle (123456) is a 4-interlacing of (153) and (264) since $(123456)^4 = (153)(264)$.

\begin{lemma}
If $k$ is a positive integer and $\sigma_1, \sigma_2, \cdots, \sigma_n$ are disjoint paths and/or cycles, then $(\sigma_1\sigma_2\cdots\sigma_n)^k = \sigma_1^k\sigma_2^k\cdots\sigma_n^k$.
\end{lemma}

\begin{lemma}
Let $a_1, a_2, \cdots, a_n$ be distinct, $k$ be a positive integer, and $n = pk + q$ where $0\leq q\leq k-1$. Then
\[
[a_1a_2\cdots a_n]^k  = \prod_{i=1}^q [a_ia_{i+k}\cdots a_{i+pk}] \prod_{i=q+1}^k [a_ia_{i+k}\cdots a_{i+(p-1)k}].
\]
In other words, $[a_1a_2\cdots a_n]^k$ is the product of $q$ paths of length $p+1$ and $k-q$ paths of length $p$.
\end{lemma}

\begin{definition}
An element $\alpha\in\SIM(n)$ is a {\em $k$th root} of the element $\sigma\in\SIM(n)$ if $\alpha^k=\sigma$. Also, $r_k(\sigma)$ denotes the number of $k$th roots $\sigma$ possesses.
\end{definition}

\begin{theorem}
The element $\alpha$ is a $k$th root of the element $\sigma$ if and only if $\alpha$ is the product of distinct, disjoint $k$-interlacings of the paths and/or cycles of the disjoint cycle and path decomposition of $\sigma$, excluding none of the paths or cycles of $\sigma$.
\end{theorem}

\begin{proof}
Let $\alpha$ be the product of disjoint $k$-interlacings of the paths and cycles of $\sigma$, excluding no paths or cycles of $\sigma$, and $\alpha=\alpha_1\alpha_2\cdots\alpha_m$ and $\sigma=\sigma_1\sigma_2\cdots\sigma_n$ be the disjoint path and cycle decompositions of $\alpha$ and $\sigma$ respectively. Since $\alpha_1,\cdots,\alpha_m$ are disjoint, $\alpha_1^k,\alpha_2^k,\cdots,\alpha_m^k$ must be disjoint as well; hence, no paths or cycles of $\sigma$ are repeated. Since $\alpha^k$ excludes no paths or cycles of $\sigma$, $\alpha^k=\alpha_1^k\alpha_2^k\cdots\alpha_m^k=\sigma_1\sigma_2\cdots\sigma_n=\sigma$. Therefore, $\alpha$ is a $k$th root of $\sigma$.

Now let $\alpha$ be a $k$th root of $\sigma$ and assume $\alpha$ is not the product of disjoint $k$-interlacings of the paths or cycles of $\sigma$, excluding no paths or cycles of $\sigma$. Then either there exists an $i$ such that $\alpha_i$ is not a $k$-interlacing of any paths or cycles of $\sigma$, or $\alpha^k$ does not include all paths or cycles of $\sigma$. In the latter case, $\alpha^k\neq\sigma$ as $\alpha^k$ does not have the same paths or cycles as $\sigma$ in their disjoint path and cycle decompositions, contradicting $\alpha$ being a $k$th root of $\sigma$. So there exists an $\alpha_i$ that is not a $k$-interlacing of some of the paths or cycles of $\sigma$. By this, $\alpha_i^k$ is not the product of {\em only} paths or cycles of $\sigma$. Since $\alpha^k$ contains a path or cycle not in $\sigma$, $\alpha^k\neq\sigma$. Therefore, $\alpha$ is not a $k$th root of $\sigma$, a contradiction. In total, $\alpha$ must be the product of disjoint $k$-interlacings of the paths and cycles of $\sigma$, excluding no paths or cycles of $\sigma$.  
\end{proof}

\begin{definition}
Let $P=\{\sigma_1,\cdots,\sigma_n\}$ be a nonempty set of disjoint paths and/or cycles. Then denote the number of $k$-interlacings of $\sigma_1,\cdots,\sigma_n$ by $\phi_k(P)$. If none exist, then $\phi_k(P)=0$.
\end{definition}

\begin{Corollary}
Let $\sigma_1\sigma_2\cdots\sigma_n$ be the disjoint path and cycle decomposition of $\sigma$, and $P = \{\sigma_1,\sigma_2,\cdots,\sigma_n\}$. Then 
\[
r_k(\sigma) = \sum_{\cup\{P_1,\cdots, P_m\}=P}\prod_{i=1}^m \phi_k(P_i)
\]
where the sum is taken over all set partitions $\{P_1,\cdots,P_m\}$ of $P$.
\end{Corollary}

\begin{proof}
It follows from Theorem 1 that the number of $k$th roots of $\sigma$ is the number of ways to construct products of disjoint $k$-interlacings of the paths and/or cycles of $\sigma$, excluding no paths or cycles of $\sigma$. Each way to interlace the paths and/or cycles can be given by a set partition of $P$, where each set in a partition of $P$ may or not be able to be interlaced. The number of ways to interlace the paths and/or cycles for a set partition $\{P_1,\cdots,P_m\}$ is then given by $\prod_{i=1}^m\phi_k(P_i)$. Note that when a set partition of $P$ contains a set that cannot be interlaced, then the product becomes zero.
\end{proof}

\begin{theorem}
Let $\sigma$ have a disjoint path and cycle decomposition of the form $\sigma = \alpha\beta$, where $\alpha$ is the product of disjoint cycles and $\beta$ is the product of disjoint paths. Also let $A$ be the set of all the digits of the cycles of $\alpha$ and $B$ be the set of all the digits of the paths of $\beta$. Then the number of $k$th roots of $\sigma$ is the product of the number of $k$th roots of $\alpha\in S_A$ and the number of $k$th roots of $\beta\in\SIM(B)$.
\end{theorem}

\begin{proof}
Since paths result in paths and cycles result in cycles when risen to the $k$th power, paths cannot be interlaced with cycles. By Corollary 1, it follows that the number of $k$th roots of $\sigma$ is the product of the number of ways to interlace all the cycles of $\alpha$ and the number of ways to interlace all the paths of $\beta$. By Theorem 1, this is the same as the product of the number of $k$th roots of $\alpha\in S_A$ and the number of $k$th roots of $\beta\in\SIM(B)$.
\end{proof}

From previous research on enumerating $k$th roots in $S_n$ and Theorem 2, we are able to just focus our attention to interlacing paths. For formulas on enumerating $k$th roots in $S_n$, see \cite{kn:Lea} and \cite{kn:Pav}.

\begin{lemma}
The disjoint paths $\sigma_1,\sigma_2,\cdots,\sigma_n$, not all of length one, have a $k$-interlacing if and only if $|\ell(\sigma_i)-\ell(\sigma_j)|\leq 1$ for all $i\neq j$, and $n=k$. Also, if there are $m$ paths in $P=\{\sigma_1,\sigma_2,\cdots,\sigma_k\}$ of length $\ell$ where $0\leq m<k$ and there are $k-m$ paths of length $\ell-1$, then $\phi_k(P)=m!(k-m)!$.
\end{lemma}

\begin{proof}
Let the paths $\sigma_1,\sigma_2,\cdots,\sigma_n$ have a $k$-interlacing $\alpha$. By definition, $\alpha^k = \sigma_1\sigma_2\cdots\sigma_n$, and by Lemma 3, $\alpha^k$ is the product of $k$ disjoint paths of length varying by at most one; let these paths be denoted $\beta_1,\beta_2,\cdots,\beta_k$. Hence, $\alpha^k = \sigma_1\sigma_2\cdots\sigma_n = \beta_1\beta_2\cdots\beta_k$, and $\sigma_1,\sigma_2,\cdots,\sigma_n$ must be the same paths as $\beta_1,\beta_2,\cdots,\beta_k$. Therefore, $n=k$ and $\sigma_1,\sigma_2,\cdots,\sigma_n$ vary in length by at most one.

Now let $\sigma_i=[a_{i1}a_{i2}\cdots a_{i(\ell+1)}]$ if $i\leq c$ and $\sigma_i=[a_{i1}a_{i2}\cdots a_{1\ell}]$ if $c<i\leq k$. Then, $\alpha=[a_{11}a_{21}\cdots a_{k1}a_{12}a_{22}\cdots a_{k2}\cdots a_{1(\ell+1)}a_{2(\ell+1)}\cdots a_{c(\ell+1)}]$ is a $k$-interlacing of $\sigma_1,\cdots,\sigma_n$ from Lemma 3.

Let $P=\{\sigma_1,\cdots,\sigma_n\}$, where $P$ has $m$ paths of length $\ell$ and $n-m$ paths of length $\ell-1$. It will be shown that $\phi_k(P)=m!(n-m)!$. In order for a path $\alpha$ to be a $k$-interlacing of $\{\sigma_1,\cdots,\sigma_n\}$, $\alpha$ has to be of the form $\alpha=[a_1a_2\cdots a_{k(\ell-1)+m}]$ where $a_i$ are all digits of $\sigma_1,\cdots,\sigma_n$, without any digits repeated. By Lemma 3, $\alpha^k=[a_1a_{1+k}\cdots a_{1+k(\ell-1)}]\cdots$ $[a_m\cdots a_{m+k(\ell-1)}]$ $[a_{m+1}\cdots a_{m+k(\ell-2)}]\cdots$ $[a_k\cdots a_{k+k(\ell-2)}]$. So, $a_1,a_{1+k},\cdots,a_{1+k(\ell-1)}$ are digits of $\sigma_i$ for one $i$, and likewise for the digits of the other paths of $\alpha$. Hence, $a_1,a_2,\cdots,a_k$ are all from different paths, and if $a_i$ is a digit in a path, then $a_{i+k}$, if it exists, must be in the same path as $a_i$. It then follows that $a_1,\cdots,a_m$ have to be from the paths of length $\ell$ and $a_{m+1},\cdots,a_n$ from the paths of length $\ell-1$. Since the path containing $a_i$ where $i>k$ is determined from the path containing $a_{i-k}$, only the choice in ordering $a_1,\cdots,a_m$ and then $a_{m+1},\cdots,a_k$ matters. Therefore, $\phi_k(P)=m!(k-m)!$.

\end{proof}

\begin{lemma}
Let $\sigma_1,\sigma_2,\cdots,\sigma_n$ be disjoint paths of length one, and \sloppy{$P=\{\sigma_1,\cdots,\sigma_n\}$}. Then there exists a $k$-interlacing of these paths if and only if $n\leq k$. Also, $\phi_k(P)=n!$.
\end{lemma}

\begin{proof}
Suppose $\alpha$ is a $k$-interlacing of $\sigma_1,\sigma_2,\cdots,\sigma_n$. Then, \sloppy{$\ell(\alpha)=\sum_i\ell(\sigma_i)=n\leq k$}. Now suppose $\sigma_1=[a_1]$, $\sigma_2=[a_2],\cdots$, $\sigma_n=[a_n]$ where $n\leq k$, and let $P=\{\sigma_1,\cdots,\sigma_n\}$. Then, $[a_1a_2\cdots a_n]$ is a $k$-interlacing of $\sigma_1,\cdots,\sigma_n$. Now it will be shown that $\phi_k(P)=n!$. Any path $\alpha$ of the form $[a_1a_2\cdots a_n]$, where $a_1,a_2,\cdots,a_n$ are distinct digits of $\sigma_1,\cdots,\sigma_n$, is a $k$-interlacing of $\{\sigma_1,\cdots,\sigma_n\}$. Since there are $n!$ of these paths, $\phi_k(P)=n!$.
\end{proof}

\section{Enumerating $k$th Roots}

Interestingly, integer partitions serve a significant purpose in determining the number of $k$th roots an element of $\SIM(n)$ has. Hence, the notation for integer partitions used in the upcoming theorems will be explained below.

\begin{definition}
Let $\lambda: a_1+\cdots+a_t = n$ be an integer partition of $n$ where $a_1,\cdots,a_t$ all are positive integers. Then we shall write $\lambda\vdash n$. If $a$ is one of $a_1,\cdots,a_t$, then write $a\in\lambda$. Let $\freq(a,\lambda)$ denote the number of times $a$ appears in $a_1,\cdots,a_t$. Also, let $\max\lambda = \max_{1\leq i\leq t}a_i$, and $\ell(\lambda)$, the ``length'' of $\lambda$, equal $t$. Finally, let $\sum\lambda = \sum_i a_i = n$.
\end{definition}

\begin{proposition}
Let $\sigma$ have disjoint path decomposition of $\sigma_1\sigma_2\cdots\sigma_n$ where $\ell(\sigma_1)=\cdots=\ell(\sigma_n)>1$. Then $\sigma$ has a $k$th root if and only if $k|n$, and if $k|n$, then $\sigma$ has $n!/(n/k)!$ $k$th roots.
\end{proposition}

\begin{proof}
Let $P = \{\sigma_1,\cdots,\sigma_n\}$. From Corollary 1, $\sigma$ has $\sum\prod_{i=1}^m\phi_k(P_i)$ $k$th roots, with the sum taken over all set partitions $\{P_1,\cdots,P_m\}$ of $P$. By Lemma 4, $\phi_k(P_i)\neq 0$ if and only if the cardinality of $P_i$, $|P_i|$, is $k$. Also, if $|P_i|=k$, then $\phi(P_i)=k!$. Hence, for any set partition $\{P_1,\cdots,P_m\}$ of $P$ such that $\prod_{i=1}^m\phi_k(P_i)\neq 0$, $|P_i|=k$ for each $i$; $\prod_{i=1}^m\phi_k(P_i) = \prod_{i=1}^{n/k}k! = (k!)^{n/k}$. Since there are

\[\frac{1}{(n/k)!}\binom nk\binom{n-k}{k}\cdots\binom kk = \frac{n!}{(k!)^{n/k}(n/k)!} \]

set partitions such that $\prod_{i=i}^m\phi_k(P_i)\neq 0$, $\sigma$ has $\frac{n!}{(k!)^{n/k}(n/k)!}(k!)^{n/k} = \frac{n!}{(n/k)!}$ $k$th roots.
\end{proof}

\begin{proposition}
Let $\sigma$ have disjoint path decomposition of $\sigma_1\sigma_2\cdots\sigma_n$ where $\ell(\sigma_1)=\cdots=\ell(\sigma_n)=1$. Then

\[ r_k(\sigma) = \sum_\lambda\frac{n!}{\prod_{a\in\lambda}\freq(a,\lambda)!} \]

\noindent where the sum is taken over all integer partitions $\lambda\vdash n$ such that $\max\lambda\leq k$.
\end{proposition}

\begin{proof}
Let $P = \{\sigma_1,\cdots,\sigma_n\}$. Then by Corollary 1, $r_k(\sigma) = \sum\prod_{i=1}^m\phi_k(P_i)$ where the sum is taken over all set partitions $\{P_1,\cdots,P_m\}$ of $P$. By Lemma 5, a set partition $\{P_1,\cdots,P_m\}$ of $P$ satisfies $\prod_{i=i}^m\phi_k(P_i)\neq 0$ if and only if $|P_i|\leq k$. Suppose $\{P_1,\cdots,P_m\}$ and $\{Q_1,\cdots,Q_m\}$ are set partitions of $P$. It will be shown that if the integer partitions $\lambda_P: |P_1|+\cdots+|P_m| = n$ and $\lambda_Q: |Q_1|+\cdots+|Q_m| = n$ are the same, then $\prod_{i=i}^m\phi_k(P_i)=\prod_{i=i}^m\phi_k(Q_i)$. Suppose $|P_1|+\cdots+|P_m| = n$ and $|Q_1|+\cdots+|Q_m| = n$ are the same partitions. Then, 

\[ \prod_{i=i}^m\phi_k(P_i) = \prod_{i=1}^m|P_i|! = \prod_{i=1}^m|Q_i|! = \prod_{i=1}^m\phi_k(Q_i) \]

\noindent since $|P_1|,\cdots,|P_m|$ is simply $|Q_1|,\cdots,|Q_m|$ rearranged. Also, every integer partition $a_1+\cdots+a_m = n$ such that $a_i\leq k$ for each $i$ can represent set partitions $\{P_1,\cdots,P_m\}$ such that $|P_i|=a_i$ for each $i$. Given an integer partition $a_1+\cdots+a_m = n$, there are

\[ \frac{1}{\prod_{a\in\lambda}\freq(a,\lambda)!}\binom{n}{a_1}\binom{n-a_1}{a_2}\binom{n-a_1-a_2}{a_3}\cdots
\binom{a_m}{a_m} \]
\[ =\frac{n!}{a_1!a_2!\cdots a_m!\prod_{a\in\lambda}\freq(a,\lambda)!} \]

\noindent set paritions $\{P_1,\cdots,P_m\}$ of $P$ such that $|P_i|=a_i$ for all $i$. Since \sloppy{$\prod_{i=i}^m\phi_k(P_i) = \prod_{i=1}^m|P_i|! = a_1!\cdots a_m!$},

\begin{align*}
r_k(\sigma) &= \sum_{\lambda\vdash n, \max\lambda\leq k}\frac{n!}{a_1!\cdots a_m!\prod_{a\in\lambda}\freq(a,\lambda)!}a_1!\cdots a_m! \\
 &= \sum_{\lambda\vdash n, \max\lambda\leq k}\frac{n!}{\prod_{a\in\lambda}\freq(a,\lambda)!}
\end{align*}
\end{proof}

\begin{proposition}
Let $\sigma_{11}\cdots\sigma_{1n_1}\sigma_{21}\cdots\sigma_{2n_2}$ be the disjoint path decomposition of $\sigma$ where $\ell(\sigma_{11})=\cdots=\ell(\sigma_{1n_1})=\ell(\sigma_{21})+1=\cdots=\ell(\sigma_{2n_2})+1>2$ and let $n_1+n_2=n$. Then

\[ r_k(\sigma) = \sum_\lambda\frac{n_1!n_2!}{(n/k-\ell(\lambda))!\prod_{a\in\lambda}\freq(a,\lambda)!} \]

\noindent where the sum is taken over all integer partitions $\lambda\vdash n_1$ such that $\max\lambda\leq k$ and $\ell(\lambda)\leq n/k$.
\end{proposition}

\begin{proof}
By Corollary 1, $r_k(\sigma) = \sum_{\cup\{P_1,\cdots,P_m\}=P}\prod_{i=1}^m\phi_k(P_i)$, where \sloppy{$P = \{\sigma_1,\cdots,\sigma_n\}$}. By Lemma 4, a set partition $\{P_1,\cdots,P_m\}$ of $P$ has to satisfy $|P_i|=k$ for all $i$ in order for $\prod_{i=1}^m\phi_k(P_i)\neq 0$. Hence, $k|n$. Define the function $N(P_i)$ by $N(P_i)=|\{\alpha\in P_i|\ell(\alpha)=\ell(\sigma_1)\}|$. Suppose $\{P_1,\cdots,P_{n/k}\}$ and $\{Q_1,\cdots,Q_{n/k}\}$ are set partitions of $P$. It will be shown that if the integer partitions $\lambda_P: N(P_1)+\cdots+N(P_{n/k}) = n_1$ and $\lambda_Q: N(Q_1)+\cdots+N(Q_{n/k}) = n_1$ are the same, then $\prod_{i=1}^{n/k}\phi_k(P_i)=\prod_{i=1}^{n/k}\phi_k(Q_i)$. Suppose $\lambda_P=\lambda_Q$. Then,

\[ \prod_{i=1}^{n/k}\phi_k(P_i)=\prod_{i=1}^{n/k}N(P_i)!(k-N(P_i))!=\prod_{i=1}^{n/k}N(Q_i)!(k-N(Q_i))!=\prod_{i=1}^{n/k}\phi_k(Q_i) \]

\noindent since $\lambda_P$ is $\lambda_Q$ rearranged. Suppose $\lambda: a_1+\cdots+a_m = n_1$ where $a_i\leq k$ and $m\leq n/k$. Then $\lambda$ represents

\[ \frac{1}{(n/k-m)!\prod_{a\in\lambda}\freq(a,\lambda)!}\binom{n_1}{a_1}\binom{n_2}{k-a_1}\binom{n_1-a_1}{a_2}\binom{n_2-(k-a_1)}{k-a_2} \]
\[ \cdots\binom{a_m}{a_m}\binom{k-a_m}{k-a_m} \]
\[ = \frac{n_1!}{a_1!\cdots a_m!}\frac{n_2!}{(k-a_1)!\cdots(k-a_m)!}\frac{1}{(n/k-m)!\prod_{a\in\lambda}\freq(a,\lambda)!} \]

\noindent set partitions $\{P_1,\cdots,P_{n/k}\}$ of $P$ such that $N(P_1)+\cdots+N(P_{n/k}) = n_1$ is the same partition as $\lambda$. Since $\prod_{i=1}^m\phi_k(P_i)=\prod_{i=1}^m a_i!(k-a_i)!$, we have

\begin{align*}
r_k(\sigma) &= \sum_\lambda\frac{n_1!n_2!}{\prod_{i=1}^m a_i!(k-a_i)!}\frac{\prod_{i=1}^m a_i!(k-a_i)!}{(n/k-\ell(\lambda))!\prod_{a\in\lambda}\freq(a,\lambda)!} \\
&= \sum_\lambda\frac{n_1!n_2!}{(n/k-\ell(\lambda))!\prod_{a\in\lambda}\freq(a,\lambda)!}
\end{align*}
\end{proof}

\begin{definition}
The paths $\sigma_1,\cdots,\sigma_n$ are of weakly varying length if there exists a suitable reordering $\alpha_1,\cdots,\alpha_n$ such that $\ell(\alpha_1)\leq\ell(\alpha_2)\leq\cdots\leq\ell(\alpha_n)$ and for each $i$, $|\ell(\alpha_{i+1}) - \ell(\alpha_i)|\leq 1$.
\end{definition}

\begin{theorem}
Let $\sigma = \sigma_{11}\cdots\sigma_{1n_1}\cdots\sigma_{m1}\cdots\sigma_{mn_m}$ where $\sigma_{ij}$ are disjoint paths  such that $\ell(\sigma_{11})=\cdots=\ell(\sigma_{1n_1})=\ell(\sigma_{21})+1=\cdots=\ell(\sigma_{2n_2})+1=\cdots=\ell(\sigma_{m1})+m-1=\cdots=\ell(\sigma_{mn_m})+m-1$, $m\geq 2$, and $\ell(\sigma_{ij})\geq 2$ for all $i,j$. Also let $n = n_1+\cdots+n_m$. Then
\[ r_k(\sigma) = \sum_{\lambda_1,\cdots,\lambda_{m-1}}\frac{n_1!\cdots n_m!}{(n/k-\sum_{i=1}^{n-1}\ell(\lambda_i))!\prod_{i=1}^{m-1}\prod_{a\in\lambda_i}\freq(a,\lambda_i)!} \]
where the sum is taken over integer partitions $\lambda_1,\cdots,\lambda_{m-1}$ such that $\lambda_1\vdash n_1$, $\lambda_2\vdash n_1+n_2-k\ell(\lambda_1)$, $\cdots$, $\lambda_{m-1}\vdash\sum_{i=1}^{m-1}n_i - k\sum_{i=1}^{m-2}\ell(\lambda_i)$, $\max\lambda_i\leq k$, and $\ell(\lambda_i)\leq(n_{i+1}+\sum\lambda_i)/k$ for all $i$. It is admissable for some to be partitions of zero.
\end{theorem}

\begin{proof}
This will be shown via induction. By Proposition 3, the above formula holds when $m=2$. Now suppose the formula holds for an $m>2$, and $\alpha=\alpha_{11}\cdots\alpha_{1n_1}\cdots\alpha_{m+1,1}\cdots\alpha_{m+1,n_{m+1}}$ where $\alpha_{ij}$ are disjoint paths such that $\ell(\alpha_{ij})\geq 2$ for all $i,j$ and $\ell(\alpha_{11})=\cdots=\ell(\alpha_{1n_1})=\ell(\alpha_{21})+1=\cdots=\ell(\alpha_{2n_2})+1=\cdots=\ell(\alpha_{m+1,1})+m=\cdots=\ell(\alpha_{m+1,n_{m+1}})+m$. Let $P=\{\alpha_{11},\cdots,\alpha_{m+1,n_{m+1}}\}$. From Corollary 1,
	\[ r_k(\alpha) = \sum_{\cup\{P_1,\cdots,P_r\}=P}\prod_{i=1}^r\phi_k(P_i), \]
which may be rewritten as
	\[ r_k(\alpha) = \sum_Q\sum_{\{P'_1,\cdots,P'_r\}}\prod_{i=1}^{|Q|}\phi_k(Q_i)\prod_{i=1}^r\phi_k(P'_i) \]
where each $Q$ is a collection of disjoint subsets of $P$ such that each set in $Q$ contains at least one element of $\{\alpha_{11},\cdots,\alpha_{1n_1}\}$ and $\{\alpha_{11},\cdots,\alpha_{1n_1}\}$ is contained in $\cup Q$, and $\{P'_1,\cdots,P'_r\}$ is a set partition of $P\backslash\cup Q$. Then,
	\begin{equation}
	r_k(\alpha) = \sum_Q\left(\prod_{i=1}^{|Q|}\phi_k(Q_i)\sum_{\{P'_1,\cdots,P'_r\}}\prod_{i=1}^r\phi_k(P'_i)\right)
	\end{equation}
since $\prod_{i=1}^{|Q|}\phi_k(Q_i)$ is independent of $\{P'_1,\cdots,P'_r\}$. In order to exclude terms of the above sum that are zero, take the sets of $Q$ to only have elements from $\{\alpha_{11},\cdots,\alpha_{1n_1},\alpha_{21},\cdots,\alpha_{2n_2}\}$. From Corollary 1, $\sum_{\{P'_1,\cdots,P'_r\}}\prod_{i=1}^r\phi_k(P'_i)$ is the same as the number of $k$th roots of an element that is the product of paths of $m$ different, weakly varying lengths greater than one; the multiplicities of these lengths in descending order of length are $n_1+n_2-k|Q|$, $n_3$, $\cdots$, $n_{m+1}$. By the inductive hypothesis,
	\begin{flalign*}
	&\sum_{\{P'_1,\cdots,P'_r\}}\prod_{i=1}^r\phi_k(P'_i)&
	\end{flalign*}
	\[ = \sum_{\lambda_2,\cdots,\lambda_m}\frac{(n_1+n_2-k|Q|)!n_3!\cdots n_{m+1}!}{(\frac{n-k|Q|}{k}-\sum_{i=2}^m\ell(\lambda_i))!\prod_{i=2}^m\prod_{a\in\lambda_i}\freq(a,\lambda_i)!} \]
	
\noindent where $\lambda_2\vdash n_1+n_2-k|Q|$, $\lambda_3\vdash n_1+n_2+n_3-k|Q|-k\ell(\lambda_2)$, $\cdots$, $\lambda_m\vdash\sum_{i=1}^m n_i - k|Q| - k\sum_{i=2}^{m-1}\ell(\lambda_i)$, $\max\lambda_i\leq k$ for all $i$, and $\ell(\lambda_i)\leq(n_{i+1}+\sum\lambda_i)/k$ for all $i$.

Suppose $\{Q^{(1)}_1,\cdots,Q^{(1)}_s\}$ and $\{Q^{(2)}_1,\cdots,Q^{(2)}_{s'}\}$ are valid $Q$ that can be taken over the first sum in (1), and define the function $N(Q_i)=|\{\alpha_{1i}|\alpha_{1i}\in Q_i\}|$. Then the inside of the sum in (1) for $\{Q^{(1)}_1,\cdots,Q^{(1)}_s\}$ and for \sloppy{$\{Q^{(2)}_1,\cdots,Q^{(2)}_{s'}\}$} are equal whenever $\lambda^{(1)}:N(Q^{(1)}_1)+\cdots+N(Q^{(1)}_s)=n_1$ and $\lambda^{(2)}:N(Q^{(2)}_1)+\cdots+N(Q^{(2)}_{s'})=n_1$ are the same integer partitions. Hence, for each valid $Q=\{Q_1,\cdots,Q_s\}$, the interior of the sum only depends on the associated integer partition $\lambda_Q: N(Q_1)+\cdots+N(Q_s)=n_1$. So given a partition $\lambda_1: a_1+\cdots+a_t=n_1$ with $\max\lambda_1\leq k$ and $\ell(\lambda_1)\leq(n_1+n_2)/2$, there are
	\[ \frac{1}{\prod_{a\in\lambda_1}\freq(a,\lambda_1)!}\binom{n_1}{a_1}\binom{n_2}{k-a_1}\binom{n_1-a_1}{a_2}\binom{n_2-(k-a_1)}{k-a_2} \]
	\[ \cdots\binom{a_t}{a_t}\binom{n_2-\sum_{i=1}^{t-1}(k-a_i)}{k-a_t} \]
	\[ = \frac{n_1!n_2!}{\prod_{a\in\lambda_1}\freq(a,\lambda_1)!a_1!\cdots a_t!(k-a_1)!\cdots(k-a_t)!(n_1+n_2-tk)!}\]
admissable $Q$ such that $\lambda_Q: N(Q_1)+\cdots+N(Q_s)=n_1$ and $\lambda_1: a_1+\cdots+a_t=n_1$ are the same partitions. So,
	\[ r_k(\alpha) = \sum_{\lambda_1}\left(\frac{\prod_{i=1}^t a_i!(k-a_i)!}{\prod_{a\in\lambda}\freq(a,\lambda_1)!}\frac{n_1!n_2!}{(n_1+n_2-k\ell(\lambda_1))!\prod_{i=1}^t a_i!(k-a_i)!}\right. \]
	\[ \left.\sum_{\lambda_2,\cdots,\lambda_m}\frac{(n_1+n_2-k\ell(\lambda_1))!n_3!\cdots n_{m+1}!}{(n/k-t-\sum_{i=2}^m\ell(\lambda_i))!\prod_{i=2}^m\prod_{a\in\lambda_i}\freq(a,\lambda_i)!}\right) \]
where $\lambda_1\vdash n_1$, $\lambda_2\vdash n_1+n_2-k\ell(\lambda_1)$, $\lambda_3\vdash n_1+n_2+n_3-k\ell(\lambda_1)-k\ell(\lambda_2)$, $\cdots$, $\lambda_m\vdash\sum_{i=1}^m n_i - k\sum_{i=1}^{m-1}\ell(\lambda_i)$, $\max\lambda_i\leq k$, and $\ell(\lambda_i)\leq (n_{i+1}+\sum\lambda_i)/k$ for all $i$. This simplifies to
	\[ r_k(\alpha) = \sum_{\lambda_1,\cdots,\lambda_{m}}\frac{n_1!n_2!n_3!\cdots n_{m+1}!}{(n/k-\sum_{i=1}^m\ell(\lambda_i))!\prod_{i=1}^m\prod_{a\in\lambda_i}\freq(a,\lambda_i)!}, \]
which completes the inductive step.
\end{proof}

Given the formula for $r_k(\sigma)$ where $\sigma$ is the product of disjoint paths of weakly varying length all greater than one, it is left to derive the formula for elements that are the product of disjoint paths of weakly increasing length starting with one.

\begin{theorem}
Let $\sigma = \sigma_{11}\cdots\sigma_{1n_1}\cdots\sigma_{m1}\cdots\sigma_{mn_m}$ where $\sigma_{ij}$ are disjoint paths  such that $\ell(\sigma_{11})=\cdots=\ell(\sigma_{1n_1})=\ell(\sigma_{21})+1=\cdots=\ell(\sigma_{2n_2})+1=\cdots=\ell(\sigma_{m1})+m-1=\cdots=\ell(\sigma_{mn_m})+m-1=m$, and $m\geq 2$. Also let $n = n_1+\cdots+n_m$. Then
	\[ r_k(\sigma) = \sum_{\lambda_1,\cdots,\lambda_m}\frac{n_1!\cdots n_m!}{\prod_{i=1}^{m}\prod_{a\in\lambda_i}\freq(a,\lambda_i)!} \]
where the sum is taken over integer partitions $\lambda_1,\cdots,\lambda_{m}$ such that $\lambda_1\vdash n_1$, $\lambda_2\vdash n_1+n_2-k\ell(\lambda_1)$, $\cdots$, $\lambda_{m}\vdash\sum_{i=1}^{m}n_i - k\sum_{i=1}^{m-1}\ell(\lambda_i)$, $\max\lambda_i\leq k$, and $\ell(\lambda_i)\leq(n_{i+1}+\sum\lambda_i)/k$ for all $i$. It is admissable for some to be partitions of zero.
\end{theorem}

\begin{proof}
This proof is markedly similar to that of the previous theorem; as before, induction will be used. The above formula holds for $m=1$ by Proposition 2. Now suppose the formula holds for an $m>1$, and \sloppy{$\alpha=\alpha_{11}\cdots\alpha_{1n_1}\cdots\alpha_{m+1,1}\cdots\alpha_{m+1,n_{m+1}}$} where $\alpha_{ij}$ are disjoint paths such that $\ell(\alpha_{11})=\cdots=\ell(\alpha_{1n_1})=\ell(\alpha_{21})+1=\cdots=\ell(\alpha_{2n_2})+1=\cdots=\ell(\alpha_{m+1,1})+m=\cdots=\ell(\alpha_{m+1,n_{m+1}})+m=m+1$. Let $P=\{\alpha_{11},\cdots,\alpha_{m+1,n_{m+1}}\}$. From Corollary 1 and rewriting as in the previous proof,
	\begin{equation}
	r_k(\alpha) = \sum_Q\left(\prod_{i=1}^{|Q|}\phi_k(Q_i)\sum_{\{P'_1,\cdots,P'_r\}}\prod_{i=1}^r\phi_k(P'_i)\right)
	\end{equation}
where each $Q$ is a collection of disjoint subsets of $P$ such that each set in $Q$ contains at least one element of $\{\alpha_{11},\cdots,\alpha_{1n_1}\}$ and $\{\alpha_{11},\cdots,\alpha_{1n_1}\}$ is contained in $\cup Q$, and $\{P'_1,\cdots,P'_r\}$ is a set partition of $P\backslash\cup Q$. As before, in order to exclude terms of the above sum that are zero, take the sets of $Q$ to only have elements from $\{\alpha_{11},\cdots,\alpha_{1n_1},\alpha_{21},\cdots,\alpha_{2n_2}\}$. From Corollary 1, $\sum_{\{P'_1,\cdots,P'_r\}}\prod_{i=1}^r\phi_k(P'_i)$ is the same as the number of $k$th roots of an element that is the product of paths of $m$ different, weakly varying lengths including those of length one; the multiplicities of these lengths in descending order of length are $n_1+n_2-k|Q|$, $n_3$, $\cdots$, $n_{m+1}$. By the inductive hypothesis,
	\[ \sum_{\{P'_1,\cdots,P'_r\}}\prod_{i=1}^r\phi_k(P'_i) = \sum_{\lambda_2,\cdots,\lambda_{m+1}}\frac{(n_1+n_2-k|Q|)!n_3!\cdots n_{m+1}!}{\prod_{i=2}^{m+1}\prod_{a\in\lambda_i}\freq(a,\lambda_i)!} \]
where $\lambda_2\vdash n_1+n_2-k|Q|$, $\lambda_3\vdash n_1+n_2+n_3-k|Q|-k\ell(\lambda_2)$, $\cdots$, $\lambda_{m+1}\vdash\sum_{i=1}^{m+1} n_i - k|Q| - k\sum_{i=2}^{m}\ell(\lambda_i)$, $\max\lambda_i\leq k$ for all $i$, and $\ell(\lambda_i)\leq(n_{i+1}+\sum\lambda_i)/k$ for all $i$.

 As before, for each valid $Q=\{Q_1,\cdots,Q_s\}$ on the first sum, the interior of the sum only depends on the associated integer partition $\lambda_Q: N(Q_1)+\cdots+N(Q_s)=n_1$. Hence, given a partition $\lambda_1: a_1+\cdots+a_t=n_1$ with $\max\lambda_1\leq k$ and $\ell(\lambda_1)\leq(n_1+n_2)/2$, there are
	\[ \frac{n_1!n_2!}{\prod_{a\in\lambda_1}\freq(a,\lambda_1)!a_1!\cdots a_t!(k-a_1)!\cdots(k-a_t)!(n_1+n_2-tk)!}\]
admissable $Q$ such that $\lambda_Q: N(Q_1)+\cdots+N(Q_s)=n_1$ and $\lambda_1: a_1+\cdots+a_t=n_1$ are the same partitions. So,
	\[ r_k(\alpha) = \sum_{\lambda_1}\left(\frac{\prod_{i=1}^t a_i!(k-a_i)!}{\prod_{a\in\lambda}\freq(a,\lambda_1)!}\frac{n_1!n_2!}{(n_1+n_2-k\ell(\lambda_1))!\prod_{i=1}^t a_i!(k-a_i)!}\right. \]
	\[ \left.\sum_{\lambda_2,\cdots,\lambda_{m+1}}\frac{(n_1+n_2-k\ell(\lambda_1))!n_3!\cdots n_{m+1}!}{\prod_{i=2}^{m+1}\prod_{a\in\lambda_i}\freq(a,\lambda_i)!}\right) \]
where $\lambda_1\vdash n_1$, $\lambda_2\vdash n_1+n_2-k\ell(\lambda_1)$, $\lambda_3\vdash n_1+n_2+n_3-k\ell(\lambda_1)-k\ell(\lambda_2)$, $\cdots$, $\lambda_m\vdash\sum_{i=1}^m n_i - k\sum_{i=1}^{m-1}\ell(\lambda_i)$, $\max\lambda_i\leq k$, and $\ell(\lambda_i)\leq (n_{i+1}+\sum\lambda_i)/k$ for all $i$. This simplifies to
	\[ r_k(\alpha) = \sum_{\lambda_1,\cdots,\lambda_{m+1}}\frac{n_1!n_2!n_3!\cdots n_{m+1}!}{\prod_{i=1}^{m+1}\prod_{a\in\lambda_i}\freq(a,\lambda_i)!}, \]
completing the inductive step.

\end{proof}

The previous two theorems and Theorem 2 provide the ``building blocks'' for which one may calculate the number of $k$th roots of an element of $\SIM(n)$. If an element of $\SIM(n)$ is expressed as the product of disjoint paths and cycles, one may treat the paths independent of the cycles and use the symmetric group formulae for $k$th roots to finish one part of the problem. It is then left to calculate the number of $k$th roots of the cycle-free component.

Given the cycle-free component of an element, arrange the paths in order of weakly decreasing length. The following theorem will allow the previous formulae to be used.

\begin{theorem}
Let $\sigma=\alpha_1\cdots\alpha_n$ where each $\alpha_i$ is composed of disjoint paths of weakly varying lengths such that if $\sigma_s$ is a path of $\alpha_i$ and $\sigma_t$ a path from $\alpha_j$ with $i\neq j$, then $|\ell(\sigma_s)-\ell(\sigma_t)|>1$. Then,
	\[ r_k(\sigma) = r_k(\alpha_1)\cdots r_k(\alpha_n). \]
\end{theorem}

\begin{proof}
Let $P$ be the set of all paths of $\sigma$. From Corollary 1, $ r_k(\sigma) = \sum\prod_{i=1}^r\phi_k(P_i)$, where the sum is taken over set partitions $\{P_1,\cdots,P_r\}$ of $P$. In order to exclude some of the vanishing terms of the sum, sum over the set partitions of $P$ such that each $P_j$ in the set partition contains only paths from one of $\alpha_1,\cdots,\alpha_n$.

Now let $Q_1,\cdots,Q_n$ be sets of the paths of $\alpha_1,\cdots,\alpha_n$ respectively. By removing the vanishing terms as above, each set partition of $P$ is the union of set partitions of $Q_i$ for all $i$. Since $Q_1,\cdots,Q_n$  are disjoint, the formula from Corollary 1 becomes the product of the sums of each set partition of each $Q_i$, which becomes $ r_k(\sigma) = r_k(\alpha_1)\cdots r_k(\alpha_n)$.

\end{proof}

\section{Discussion}

Through this research, formulae were developed for the number of $k$th roots an arbitrary element in $\SIM(n)$ has. Some questions immediately follow from these results. For one, could these formulae be directly simplified, possibly by the use of integer partition identities? Also, might there be other means of computing $r_k(\sigma)$? Coming up with equivalent but different formulae for $r_k(\sigma)$ could result in some interesting identities.

The problem of enumerating $k$th roots, as well as the probability of possessing a $k$th root, may be posed for various other semigroups and groups. For instance, in \cite{kn:Ann1}, Annin et al. ask these questions for the alternating group and other subgroups of $S_n$, and similarly in \cite{kn:Ann2}, Annin et al. ask for necessary and sufficient conditions for possessing a $k$th root for elements in submonoids of $\SIM(n)$. Once one finds these conditions for $k$th roots, the enumeration problem may be asked for these submonoids.

In \cite{kn:Lea}, Lea\~nos et al. give generating formulae for the number of $k$th roots of an arbitrary element of $S_n$. Given this, what are some generating formulae for the number of $k$th roots of elements in $\SIM(n)$? Also, what are some generating formulae for the number of elements in $S_n$ that have a $k$th root?

In \cite{kn:Pav}, Pavlov derives asymptoticly equivalent formulae for the number of elements in $S_n$ that possess exactly one $k$th root. As before, one may ask for asymptotic formulae for quantities involving $k$th roots (e.g. probability of having $k$th root, $r_k(\sigma)$, etc.) in $\SIM(n)$.

The questions posed in the previous paragraphs asked for, in one form or another, simplified or alternate forms of the formulae given in this paper. Beyond these equations, there is still far more to be asked. What are some relationships among $k$th roots of an element in $\SIM(n)$? Given a graphical representation of $\SIM(n)$, possibly by means similar to Cayley graphs for groups, where do the $k$th roots of an element lie on the graph with respect to one another?

\section{Acknowledgements}
This research was funded by Lamar University's Office of Undergraduate Research Grant 2014-15 and the Ronald E. McNair Post Baccalaureate Achievement Program. Special thanks go to Dr. Valentin Andreev for mentoring and advising this project.

\noindent {\bf Christopher W. York}

\noindent Department of Mathematics

\noindent Lamar University

\noindent P.O. Box 10047

\noindent Beaumont, TX 77710

\noindent Email: {\tt cyork2@lamar.edu}

\end{document}